\documentclass{article}

\usepackage{amsmath,amsfonts,amssymb}

\usepackage{amsthm}
\usepackage[skip=5pt plus1pt, indent=10pt]{parskip}

\newcommand{\Q}{\mathbb{Q}}
\newcommand{\Z}{\mathbb{Z}}
\newcommand{\N}{\mathbb{N}}
\newcommand{\C}{\mathbb{C}}
\newcommand{\R}{\mathbb{R}}

\newcommand\set[1]{\left\{#1\right\}}
\newcommand\br[1]{\left(#1\right)}
\newcommand\abs[1]{\left|#1\right|}
\newcommand\norm[1]{\lVert#1\rVert}

\newtheorem{theorem}{Theorem}
\newtheorem{lemma}{Lemma}
\newtheorem{remark}{Remark}
\newtheorem{proposition}{Proposition}

\DeclareMathOperator{\Hom}{Hom}
\DeclareMathOperator{\rank}{rank}

\usepackage[
	backend=biber,
	style=alphabetic,
	sorting=nyt
]{biblatex}

\begin{document}

    \title{Twisted Thue  equations with multiple exponents in fixed number fields}
	
	\author{
		Tobias Hilgart \small tobias.hilgart@plus.ac.at, \\
        Volker Ziegler \small volker.ziegler@plus.ac.at, \\ 
		University of Salzburg (Austria) \\			
	}
	
	\maketitle
	
	\begin{abstract}
		Let $K$ be a number field of degree $d\geq 3$ and fix $s$ multiplicatively independent $\gamma_1, \dots, \gamma_s \in K^*$ that fulfil some technical requirements, which can be vastly simplified to $\Q$-linearly independence, given Schanuel's conjecture. We then consider the twisted Thue equation
		\[
			\abs{N_{K/\Q}\br{X-\gamma_1^{t_1}\cdots\gamma_s^{t_s}Y}} = 1,
		\]
		and prove that it has only finitely many solutions $\br{x,y, \br{t_1, \dots, t_s} }$ with $xy \neq 0$ and $\Q\br{ \gamma_1^{t_1}\cdots \gamma_s^{t_s} } = K$, all of which are effectively computable.
	\end{abstract}

    \section{Introduction}

        One of the first non-binary parametrised Thue equations to ever be solved was 
        \[
            X^3-\br{a-1}X^2Y-\br{a+2}XY^2-Y^3 = \pm 1,
        \]
        done so by Thomas \cite{Tho90}. If we express the polynomial equation in terms of its roots $\alpha_1, \alpha_2, \alpha_3$, which of course depend on the parameter $a$, then we can write this equivalently as
        \[
            \br{X-\alpha_1 Y}\br{X - \alpha_2 Y}\br{X - \alpha_3 Y} = \pm 1,
        \]
        or $\abs{ N_{K/\Q}\br{ X- \alpha_1 Y } } = 1$, where $K = \Q\br{\alpha_1}$. This result was extended by Levesque and Waldschmidt \cite{LeWa15} by twisting the equation by an exponential parameter $t$ and showed that $\abs{ N_{K/\Q}\br{X-\alpha_1^t Y} } = 1$ still only has finitely many integer solutions in $x,y,a,t$ where $\abs{y} > 1$. They conjectured that this still holds, if the Thue equation is twisted by two exponents, i.e. $\abs{ N_{K/\Q}\br{X-\alpha_1^t\alpha_2^s Y} } = 1$, and it is reasonable to expect that if it works for two, it might work for finitely many -- at least given sufficiently "nice" conditions.

        In a similar but different vein, Levesque and Waldschmidt \cite{LeWa13} also showed the following: Let $K = \Q\br{\alpha}$, with embeddings $\Phi = \set{ \Tilde{\sigma}_1, \dots, \Tilde{\sigma}_d }$ into $\C$ and $0 < \nu < 1$. Furthermore, let $\varepsilon \in \Z_K$ be an integral unit that fulfils $\Q\br{\alpha \varepsilon} = K$ and 
        \begin{equation}\label{cond:wald-lev}
            \abs{ \Tilde{\sigma}_1(\alpha\varepsilon) } = \max_i\br{ \abs{ \Tilde{\sigma}_i(\alpha \varepsilon) } }, \; \abs{ \Tilde{\sigma}_2(\alpha \varepsilon) } \geq \max_i\br{ \abs{ \Tilde{\sigma}_i(\alpha\varepsilon) } }^\nu 
        \end{equation}
        for two distinct embeddings $\Tilde{\sigma}_1,\Tilde{\sigma}_2 \in \Phi$. Then for any solution $\br{x,y,\varepsilon}$ of the inequality
        \[
            \abs{ N_{K / \Q}\br{ x - \alpha\varepsilon \, y } } \leq m
        \]
        the logarithms $\log\abs{x}, \log\abs{y}$, as well as the absolute logarithmic height $h(\alpha \varepsilon)$ can be bounded by $m^c$ for some effective constant $c$. They conjectured that the result would hold even without Condition \eqref{cond:wald-lev}. Proof of the finiteness of the number of solutions is already established in \cite{LeWa12} but rests on Schmidt's subspace theorem which does not allow for an effective upper bound for their heights.

        For our result, we also fix the number field. The "base" elements of the norm-form equation are fixed as well but can be chosen somewhat more freely. Our main result is as follows:
        
        \begin{theorem}\label{thm:main}
            Let $K$ be a number field of degree $d \geq 3$ and $s \leq d-2$. Let $\gamma_1, \dots, \gamma_s \in K^*$ be multiplicatively independent algebraic integers such that for each choice of $d-1$ embeddings $\Tilde{\sigma}_1, \dots, \Tilde{\sigma}_{d-1} \in \Hom_\Q\br{K, \C}$, we have
            \begin{equation}\label{eq:requirement}
                \rank
                \begin{pmatrix}
                    \log\abs{ \frac{ \Tilde{\sigma}_1(\gamma_1) }{ \Tilde{\sigma}_{d-1}(\gamma_1) } } & \cdots & \log\abs{ \frac{ \Tilde{\sigma}_1(\gamma_s) }{ \Tilde{\sigma}_{d-1}(\gamma_s) } } \\
                    \vdots & \ddots & \vdots \\
                    \log\abs{ \frac{ \Tilde{\sigma}_{d-2}(\gamma_1) }{ \Tilde{\sigma}_{d-1}(\gamma_1) } } & \cdots & \log\abs{ \frac{ \Tilde{\sigma}_{d-2}(\gamma_s) }{ \Tilde{\sigma}_{d-1}(\gamma_s) } }
                \end{pmatrix}
                = s, \tag{$*$}
            \end{equation}
            i.e. the matrix has full column rank $s$. Then the Thue equation
            \begin{equation}\label{eq:main}
                \abs{N_{K/\Q}\br{X-\gamma_1^{t_1}\cdots\gamma_s^{t_s}Y}} = 1    
            \end{equation}
            has only finitely many integer solutions $\br{x, y, \br{t_1, \dots, t_s}} \in \Z^2 \times \N^s$, where $xy \neq 0$ and $\Q\br{ \gamma_1^{t_1}\cdots \gamma_s^{t_s} } = K$, all of which can be effectively computed.
        \end{theorem}
        \begin{remark}
            Given Schanuel's conjecture \cite{Lan66}, that for any $\Q$-linearly independent complex numbers $z_1, \dots, z_n$ the field extension $\Q\br{z_1,\dots, z_n, e^{z_1}, \dots, e^{z_n}}$ has transcendence degree at least $n$ over $\Q$, our condition on the $\gamma_i$ can be relaxed in the following way:
    
            Instead of requiring the full rank of the given matrices, it suffices that $\gamma_1, \dots, \gamma_s$ are $\Q$-linearly independent. See section 4 for more details.
        \end{remark}

        The outline of the proof is that we follow the standard Baker's method to construct a linear form in logarithms and use lower bounds for linear forms in logarithms \cite{BaWh93} to derive an absolute upper bound for $\abs{t_1}, \dots, \abs{t_s}$. Using Bugeaud's and Gy\H{o}ry's explicit upper bound for solutions of Thue equations \cite{BuGy96} allows us to bound $\log\abs{x}, \log\abs{y}$ in terms of $\abs{t_1}, \dots, \abs{t_s}$ and thus absolutely as well. 

        For this to work, however, we need that the embeddings of $\gamma_1^{t_1} \cdots \gamma_s^{t_s}$ are not all "close", asymptotically in terms of $\max_i\abs{t_i}$, to a distinguished embedding, which is given by the type $j$ of the solution. If they are, we need a different argument, for which we require the condition \eqref{eq:requirement} or Schanuel's conjecture.

    \section{Preliminaries}
        We start by listing the results of Baker and Wüstholz \cite{BaWh93}, or Bugeaud and Gy\H{o}ry \cite{BuGy96} respectively, as well as two smaller lemmata, one of which goes back to Tijdeman \cite{Tijd73} and is stated in a slightly different setting.

        For the sake of completeness, we briefly define the absolute (or Weil) height and Mahler's measure, see, for instance, \cite{Bomb82} or \cite{Smyth08}. If $K$ is a number field of degree $d = [K:\Q]$, and for every place $\nu$, we write $d_\nu = [K_\nu : \Q_\nu]$, then we normalise the absolute value $\abs{\cdot}_\nu$ so that
        \begin{enumerate}
            \item if $\nu \vert p$ for a prime number $p$, then $\abs{p}_\nu = p^{-d_\nu / d}$,
            \item if $\nu \vert \infty$ and $\nu$ is real, then $\abs{x}_\nu = \abs{x}^{1/d}$,
            \item if $\nu \vert \infty$ and $\nu$ is complex, then $\abs{x}_\nu = \abs{x}^{2/d}$,
        \end{enumerate}
        and $\abs{x}$ denotes the Euclidian absolute value in $\R$ or $\C$. In view of this normalisation, the product formula
        \[
            \prod_{ \nu } \abs{\alpha}_\nu = 1
        \]
        holds for every $\alpha \in K^*$. The absolute height of $\alpha \in K$ is then defined as
        \[
            H(\alpha) = \prod_\nu \max\br{ 1, \abs{\alpha}_\nu},
        \]
        and the absolute logarithmic height as $h(\alpha) = \log H(\alpha)$. The absolute height is then equal to the Mahler measure $M(m_\alpha)$ of its minimal polynomial $m_\alpha$, i.e. if $m_\alpha(X) = a_n\prod_{i=1}^n \br{X - \alpha_i} \in \Z[X]$ is the minimal polynomial of $\alpha \in K$, then
        \[
            h(\alpha) = \frac{1}{n} \log M(m_\alpha) = \frac{1}{n} \br{ \abs{a_n} + \sum_{i=1}^n \log\max\br{1, \abs{\alpha_i}} }.
        \]

        \begin{proposition}[\cite{BaWh93}]\label{prop:Baker}
            Let $\gamma_1, \dots, \gamma_t$ be algebraic numbers not $0$ or $1$ in $K = \Q\br{\gamma_1,\dots, \gamma_t}$, which is of degree $D$. Let $b_1, \dots, b_t \in \Z$ and
            \[
                \Lambda = b_1 \log\gamma_1 + \cdots + b_t \log\gamma_t \neq 0.
            \]

            Then 
            \[
                \log\abs{ \Lambda } \geq - C \cdot h_1\cdots h_t \cdot \log B,
            \]
            where $C = 18(t+1)!t^{t+1}(32D)^{t+2}\log(2tD)$, $B \geq \max\br{3, \abs{b_1}, \cdots, \abs{b_t} }$ and
            \[
                h_i \geq \max\br{ h(\gamma_i), \log\abs{\gamma_i}\, D^{-1}, 0.16\, D^{-1}  }
            \]
            for $i \in \set{1, \dots, t}$.
        \end{proposition}

        \begin{proposition}[\cite{BuGy96}]\label{prop:BuGy}
            Let $B \geq \max\br{ \abs{m}, e }$, $f$ be an irreducible polynomial with root $\alpha$ and $K = \Q\br{\alpha}$. Let $R$ be the regulator of $K$ and $r$ be the unit rank. Let $H$ be an upper bound to the absolute values of the coefficients of $f$ and $n = \deg f \geq 3$. Let $F(X,Y) = Y^n f\br{\frac{X}{Y}}$, then all solutions $(x,y) \in \Z^2$ of the Thue equation $F(X,Y) = m$ satisfy
            \[
                \log\max\br{ \abs{x}, \abs{y} } \leq c \cdot R \cdot \max\br{ \log R, 1 } \br{ R + \log\br{HB} },
            \]
            where $c = 3^{r+27}(r+1)^{7r+19}n^{2n+6r+14}$.
        \end{proposition}

        The next lemma goes back to a result of Tijdeman \cite{Tijd73} on the difference of consecutive numbers that are comprised of primes up to some given bound. We can adapt the statement to our setting with only small alterations and prove it almost analogously.

        \begin{lemma}\label{lem:tijd}
            Let $K$ be a number field of degree $d \geq s$ and $\gamma_1, \dots, \gamma_s \in K^*$ multiplicatively independent. Let $\gamma = \gamma(t_1, \dots, t_s) = \gamma_1^{t_1} \cdots \gamma_s^{t_s}$ for non-zero integers $t_1, \dots, t_s$.

            Then for any two conjugates $\gamma^{(1)}, \gamma^{(2)}$ of $\gamma$ with $M = \abs{ \gamma^{(1)} } > \abs{ \gamma^{(2)} } = m$ there exists an effectively computable constant $c$ independent of $t_1, \dots, t_s$ such that

            \[
                M - m > \frac{ M }{ h(M)^c }.
            \]
        \end{lemma}
        \begin{proof}
            We start with moving the conjugation of $\gamma$ down to the conjugations of the individual $\gamma_i$, i.e. if $\gamma^{(1)}$ is the conjugation of $\gamma$ under the embedding $\Tilde{\sigma}$, let $\gamma_i^{(1)} = \Tilde{\sigma}\br{\gamma_i}$, then we write
            \[
                M - m = \abs{ \gamma^{(1)} } - \abs{ \gamma^{(2)} } = \abs{ \br{ \gamma_1^{(1)} }^{t_1} \cdots \br{ \gamma_s^{(1)} }^{t_s} } - \abs{ \br{ \gamma_1^{(2)} }^{t_1} \cdots \br{ \gamma_s^{(2)} }^{t_s} },
            \]
            and thus
            \[
                1 - \frac{m}{M} = 1 - \underbrace{ \abs{ \br{ \gamma_1^{(1)} }^{t_1} \cdots \br{ \gamma_s^{(1)} }^{t_s} \br{ \gamma_1^{(2)} }^{-t_1} \cdots \br{ \gamma_s^{(2)} }^{-t_s} } }_{=: e^\Lambda}.
            \]

            We want to use Proposition \ref{prop:Baker} to bound $\Lambda$. We have that $h(\gamma_i) \ll 1$. If we also had $\abs{t_i} \ll h(M)$, then we would have, by Proposition \ref{prop:Baker}, that $\abs{\Lambda} > h(M)^{-c}$ for some effectively computable $c$ independent of the $t_i$, which is precisely the constant of the proposition multiplied with the heights $h(\gamma_i)$. And if $\abs{\Lambda} > h(M)^{-c}$, then $1 - e^{\Lambda} \gg h(M)^{-c}$ which then proves the assumption after multiplying the inequality with $M$ (we can make the inequality $1-e^{\Lambda} \gg h(M)^{-c}$ a strict $">"$ by allowing for a larger constant $c$).

            So what remains is to prove that $\abs{t_i} \ll h(M)$. For that let $S$ be the set of all places $\nu$ of $K$ for which $ \abs{ \gamma^{(1)} }_\nu \neq 1$ but in any case includes all non-archimedean places $\infty_1, \dots, \infty_d$, that is $S = \set{\nu_1, \dots, \nu_n, \infty_1, \dots, \infty_d}$.

            In the spirit of an S-adic version of Minkowski spaces, we identify $\gamma$ with the $\abs{S}$-dimensional vector of logarithms of the individual valuations, i.e.
            \[
                \gamma \mapsto
                \underbrace{
                \begin{pmatrix}
                    \log\abs{\gamma}_{\infty_1} \\
                    \vdots  \\
                    \log\abs{\gamma}_{\infty_d} \\
                    \log\abs{\gamma}_{\nu_1} \\
                    \vdots \\
                    \log\abs{\gamma}_{\nu_n}
                \end{pmatrix}
                }_{=: v}
                =
                t_1
                \begin{pmatrix}
                    \log\abs{\gamma_1}_{\infty_1} \\
                    \vdots  \\
                    \log\abs{\gamma_1}_{\infty_d} \\
                    \log\abs{\gamma_1}_{\nu_1} \\
                    \vdots \\
                    \log\abs{\gamma_1}_{\nu_n}
                \end{pmatrix}
                + \cdots +
                t_s
                \begin{pmatrix}
                    \log\abs{\gamma_s}_{\infty_1} \\
                    \vdots  \\
                    \log\abs{\gamma_s}_{\infty_d} \\
                    \log\abs{\gamma_s}_{\nu_1} \\
                    \vdots \\
                    \log\abs{\gamma_s}_{\nu_n}
                \end{pmatrix}.                
            \]

            We further see that the right-hand side is of the form $v = \Gamma \br{ t_1, \dots, t_s }^\top$ for the $(d+n) \times s$ dimensional matrix $\Gamma$ of the $\log\abs{\gamma_i}_\nu$. Since $\gamma_1, \dots, \gamma_s$ are multiplicatively independent, the matrix $\Gamma$ has full column-rank. We can thus multiply the equation with $ \br{ \Gamma^\top \Gamma }^{-1} \Gamma^\top$ and apply the $l_1$ norm. Using the consistency $\norm{Ax} \leq \norm{A}\cdot \norm{x}$ and hiding the matrix norms inside a constant $c$ that does not depend on the $t_i$, this gives
            \[ 
                c \cdot \abs{v}_{l_1} \geq \abs{ \begin{pmatrix} t_1 \\ \vdots \\ t_s \end{pmatrix} }_{l_1},  
            \]
            which of course implies $\abs{t_i} \ll \abs{v}_{l_1}$. For the final remaining estimate, note that to compute the height $h(\gamma) = \sum_{\nu} \max\br{ \log\abs{ \gamma }_\nu, 0}$, we can ignore all the valuations where $\abs{\gamma}_\nu = 1$ and thus sum precisely over all entries of the vector $v$ and take only the positive values. But by the product formula, we have $\sum_{i=1}^{d+n} v_i = 0$, i.e. the positive entries, which contribute to the height, and the negative that don't, cancel. This in turn means that if we sum the absolute values, which is precisely what we do when calculating the $l_1$-norm $\abs{v}_{l_1}$, it must be 2 times the sum of all positive (or negative, resp.) values, thus $\abs{v}_{l_1} = 2 h(\gamma)$. This gives us $\abs{t_i} \ll h(\gamma)$ and thus proves the assertion.
        \end{proof}

        The final preliminary lemma gives a simple way to control the product of pairwise maxima by the global maximum in a finite sequence of positive numbers.

        \begin{lemma}\label{lem:bound_prod_max}
            Let $a_1, \dots, a_d \in \R$ with $0 < a_1 \leq \cdots \leq a_d$, where $a_d > 1$ and $a_1\cdots a_d = 1$. Then we get
            \[
                \prod_{ \substack{ i\in\set{1,\dots, d} \\ i\neq j } } \max\set{a_i, a_j} \geq a_d^{\frac{1}{d-1}}
            \]
            for all fixed $j \in \setlength{1,\dots, d}$.
        \end{lemma}
        \begin{proof}
            We express everything in terms of the largest number $a_d$. The extreme case, where $a_1$ is as large as possible, is
            \[
                \frac{1}{a_d^{\frac{1}{d-1}}} \cdots \frac{1}{a_d^{\frac{1}{d-1}}} \cdot a_d = 1,
            \]
            where all terms except for $a_d$ are equal and thus have to be $a_d^{-\frac{1}{d-1}}$ for the product to still be $1$. In this situation, if $j = d$, then the maximum $\max\set{a_i, a_j}$ is $a_d$ every time and the product is $a_d^{d-1}$, which is greater than our purported bound of $ a_d^{\frac{1}{d-1}}$. 
            
            If $j < d$ instead, then the maximum is $a_d^{\frac{1}{d-1}}$ in $d-2$ cases and $a_d$ once, thus
            \[
                \prod_{ \substack{ i\in\set{1,\dots, d} \\ i\neq j } } \max\set{a_j, a_i} = \br{ \frac{1}{a_d^{\frac{1}{d-1}}} }^{d-2} a_d = a_d^{ 1 - \frac{d-2}{d-1} } = a_d^{\frac{1}{d-1}}.
            \]
            Now we move from the extreme case to the general. Let $a_i = a_d^{\frac{1}{d-1}} c_i$, where $c_1\leq\cdots \leq c_n \leq 1$, and ${ 1\leq c_{n+1} \leq \cdots \leq c_{d-1} }$. For the product $a_1\cdots a_d = 1$ to still be $1$, we need $c_1\cdots c_{d-1} = 1$ to cancel.
            
            If $n < j < d$ then it is only the constants $\geq 1$ that show up in the product $\prod_{i\neq j}\max\set{a_j,a_i}$, and we can ignore them to get the purported lower bound. If instead $j \leq n$ then we get the constant $c_j$ for the first $j$ indices, and $c_i$ afterwards. But $\prod_{i=1}^{j}c_j \prod_{i=j+1}^d c_i \geq 1$, since we substituted the possibly smaller $c_i$ for $c_j$ and the product was $1$ beforehand. Thus we can again ignore the constants and get the purported bound.
        \end{proof}        
        
    \section{Proof of the main theorem}
        Let $\br{x,y, \br{t_1,\dots t_s}} \in \Z^2 \times \Z^s$ be a solution to \eqref{eq:main} and assume that all $t_i$ are non-zero, as we would otherwise carry out the proof with $s' \leq s$ many $\gamma_i$ that actually need to be written in the equation. We also assume $xy\neq 0$, as for $y = 0$, we get a solution for $x= \pm 1$ and any choice for $(t_1, \dots, t_s)$. Similarly, for $x = 0$, it depends on the $\gamma_i$ whether there are solutions, e.g. if they are units, then $y = \pm 1$ and any $(t_1, \dots, t_s)$ would be eligible.

        Let $t := \max_{i \in \set{1,\dots, s}} \abs{t_i}$ and $\beta_i := \Tilde{\sigma_i}\br{ x - \gamma_1^{t_1}\cdots \gamma_s^{t_s} y}$. We also write $\sigma_i = \Tilde{\sigma_i}\br{ \gamma_1^{t_1}\cdots \gamma_s^{t_s} }$ and $\gamma_r^{(i)}$ for the individual embedments $\Tilde{\sigma}_i(\gamma_r)$ to make the expressions more readable. 
        
        After reshuffling the indices, we can further assume that $\abs{\sigma_1} \geq \cdots \geq \abs{\sigma_d}$. The polynomial $f(X) = N_{K/\Q}\br{X - \gamma_1^{t_1}\cdots\gamma_s^{t_s}}$ is irreducible, as $\Q\br{ \gamma_1^{t_1}\cdots \gamma_s^{t_s} } = K$ per our requirement on $(x,y,(t_1,\dots,t_s))$.
        
        We define the type of the solution to be the index $j$, for which the equation $\abs{\beta_j} = \min_{i\in \set{1,\dots, d}} \abs{\beta_i}$ holds. We distinguish between the following two cases:
        \begin{description}
            \item[Case 1] There exist at least two distinct indices $i \in \set{1,\dots, d}\backslash \set{j}$ such that $\abs{ \log\abs{ \frac{\sigma_i}{\sigma_j} } } \geq \kappa \log t$, where $\kappa$ is a (fixed but) sufficiently large constant independent of the solution $\br{x,y,\br{t_1,\dots, t_s}}$.
            \item[Case 2] For all but one index $i \in \set{1,\dots, d}\backslash \set{j}$, we have $\log\abs{ \frac{\sigma_i}{\sigma_j} } \leq \kappa \log t$.
        \end{description}

        For each case, we want to prove $t \ll 1$, i.e. that $t$ can be effectively bounded by some constant and in further consequence $\abs{x}, \abs{y} \ll 1$, which would prove the finiteness of the number of solutions to Thue Equation $\eqref{eq:main}$.

        \subsection{Case 1}
            Let $k, l$ be two distinct indices that fulfil the condition of Case 1, i.e.
            \begin{equation}\label{eq:cond}
                \abs{ \log\abs{ \frac{\sigma_k}{\sigma_j} } } \geq \kappa \log t, \; \abs{ \log\abs{ \frac{\sigma_l}{\sigma_j} } } \geq \kappa \log t  \tag{$**$}
            \end{equation}
            holds. We state Siegel's identity,
            \begin{equation}\label{eq:Siegel-alt}
                \beta_j\br{\sigma_k-\sigma_l} + \beta_l\br{\sigma_j-\sigma_k} + \beta_k\br{\sigma_l-\sigma_j} = 0
            \end{equation}
            and by dividing by the third term on the left-hand side and rearranging things slightly, this is equivalent to
            \begin{equation}\label{eq:Siegel}
                \underbrace{ \frac{\beta_j}{\beta_k} \cdot \frac{\sigma_k-\sigma_l}{\sigma_j-\sigma_l} }_{=:L} + \underbrace{ \frac{\beta_l}{\beta_k} \cdot \frac{\sigma_j-\sigma_k}{\sigma_j-\sigma_l} }_{=: L'} = 1.
            \end{equation}
            
            We now show that it is possible to choose $k,l$ so that $L$ is very small: But first, regardless of the choice of $k$ and $l$, since $\br{x,y,\br{t_1,\dots, t_s}}$ is a solution to Thue Equation \eqref{eq:main}, we have that
            \[
                1 = \abs{ N_{K/\Q}\br{x-\gamma_1^{t_1}\cdots \gamma_s^{t_s}y} } = \abs{\beta_1 \cdots \beta_d}.
            \]
            If we use the minimality of $\beta_j$, we have $2 \abs{\beta_i} \geq \abs{\beta_i - \beta_j} = \abs{y\br{\sigma_j-\sigma_i}}$ and get that
            \[
                \abs{\beta_j} = \prod_{ \substack{i\in \set{1,\dots, d} \\ i \neq j} } \frac{1}{\abs{\beta_i}} \ll \frac{1}{ \abs{y}^{d-1} \prod_{ \substack{i\in \set{1,\dots, d} \\ i \neq j} } \abs{\sigma_j-\sigma_i} }.
            \]
            
            We apply Lemma \ref{lem:tijd} on each of the $\abs{ \sigma_j - \sigma_i }$ and use the equality to the Mahler measure for the height to conclude easily that $h(\sigma_i) \ll \log\max_i\abs{\sigma_i} = \log\abs{\sigma_1}$. We thus have
            \[
                \abs{\beta_j} \ll \frac{ \log\abs{\sigma_1}^{(d-1)c} }{ \abs{y}^{d-1} \prod_{ \substack{i\in \set{1,\dots, d} \\ i \neq j} } \max\set{ \abs{\sigma_j}, \abs{\sigma_i} } }.
            \]
            
            Next, we apply Lemma \ref{lem:bound_prod_max} on the product $\prod_{ \substack{i\in \set{1,\dots, d} \\ i \neq j} } \max\set{ \abs{\sigma_j}, \abs{\sigma_i} }$, which gives
            \[
                \prod_{ \substack{i\in \set{1,\dots, d} \\ i \neq j} } \max\set{ \abs{\sigma_j}, \abs{\sigma_i} } \geq \abs{\sigma_1}^{ \frac{1}{d-1} }
            \]
            and combining this with the previous bound gives
            \[
                \abs{\beta_j} \ll \frac{ \log\abs{\sigma_1}^{(d-1)c} }{ \abs{y}^{d-1} \abs{\sigma_1}^{\frac{1}{d-1}} }.
            \]
            Plugging this into our expression $L$ and once more applying $\abs{\beta_k} \geq \frac{y}{2} \abs{\sigma_j - \sigma_k}$ finally gives
            \begin{equation}\label{eq:L-bound1}
                L = \frac{\beta_j}{\beta_k} \cdot \frac{\sigma_k-\sigma_l}{\sigma_j-\sigma_l} \ll \frac{ \log\abs{\sigma_1}^{(d-1)c} }{ \abs{y}^d \abs{\sigma_1}^{ \frac{1}{d-1} } } \cdot \frac{\abs{\sigma_k - \sigma_l}}{ \abs{\sigma_j - \sigma_k} \abs{\sigma_j - \sigma_l} } .
            \end{equation}

            The first factor already looks like it could be of order $e^{-c t}$, but we first check that the second term cannot ruin everything, at least for specific choices for $k,l$.
            
            We start by noting that since $x \neq 0$, we have $\abs{x} \geq 1$. We also have that $\abs{\beta_j} = \abs{x - \sigma_j y } \leq q < 1$ for some $q$ and thus $1-q \leq \abs{x} - q \leq \abs{\sigma_j y}$, or
            \begin{equation}\label{eq:lb_y_sigj}
                \abs{y}^{-1} \ll \abs{\sigma_j}.
            \end{equation}
            
            Furthermore, we have $ \abs{\sigma_j-\sigma_k} = \abs{ \sigma_{\max} \br{ 1 - \frac{\sigma_{\min}}{\sigma_{\max}} } } $ if we denote, by abuse of notation, $ \sigma_{\max}$ as the larger, and $\sigma_{\min}$ as the smaller of the two. But since $k$ fulfils \eqref{eq:cond}, we have $\abs{ \frac{\sigma_{\min}}{\sigma_{\max}} } \leq t^{-\kappa} $, which is less then $\frac{1}{2}$ for example, and thus
            \begin{equation}\label{eq:lb_trick}
                \abs{\sigma_j-\sigma_k} = \abs{ \sigma_{\max} \br{ 1 - \frac{\sigma_{\min}}{\sigma_{\max}} } } \gg \abs{\sigma_{\max}},
            \end{equation}
            which also holds for $l$ of course.

            We now differentiate between different cases for $j$ and show that for each one, we can choose $k,l$ that fulfil \eqref{eq:cond} and do not blow up the expression $L$.
            
            \begin{enumerate}
                \item Case $j \in \set{1,d}$:
                \begin{enumerate}
                    \item Case $j = 1$: We can choose any $k,l$ that fulfil \eqref{eq:cond}. We have $\abs{\sigma_k - \sigma_l} \leq 2 \max\br{\abs{\sigma_k}, \abs{\sigma_l}} \ll \abs{\sigma_1}$, and since $k$ fulfils \eqref{eq:cond}, we have $\abs{\sigma_1 - \sigma_k} = \abs{\sigma_1} \abs{1 - \frac{\sigma_k}{\sigma_1}} \gg \abs{\sigma_1}$ by \eqref{eq:lb_trick}, same for $l$.
                    
                    Thus, we would even further improve our first bound \eqref{eq:L-bound1} for $L$, i.e.
                    \begin{align*}
                        L \ll \frac{ \log\abs{\sigma_1}^{(d-1)c} }{ \abs{y}^d \abs{\sigma_1}^{ \frac{1}{d-1} } } \cdot \frac{\abs{\sigma_k - \sigma_l}}{ \abs{\sigma_j - \sigma_k} \abs{\sigma_j - \sigma_l} } &\ll \frac{ \log\abs{\sigma_1}^{(d-1)c} }{ \abs{y}^d \abs{\sigma_1}^{ \frac{1}{d-1} } } \cdot \frac{ \abs{\sigma_1} }{ \abs{\sigma_1} \abs{\sigma_1} } \\
                        &= \frac{ \log\abs{\sigma_1}^{(d-1)c} }{ \abs{y}^d \abs{\sigma_1}^{ 1 + \frac{1}{d-1} } }.
                    \end{align*}

                    \item Case $j = d$: We choose $k=1$, which fulfils \eqref{eq:cond}, and any other $l$ that does too. Then, $\abs{\sigma_1 - \sigma_l} \ll \abs{\sigma_1}$ on the one hand and $\abs{\sigma_d - \sigma_1} \gg \abs{\sigma_1}, \; \abs{\sigma_d-\sigma_l} \gg \abs{\sigma_d}$ by \eqref{eq:lb_trick} on the other, while $\abs{\sigma_d} \gg \abs{y}^{-1}$ by \eqref{eq:lb_y_sigj}.

                    Thus we worsen our first bound \eqref{eq:L-bound1} for $L$ by a factor $\abs{y}$,
                    \[
                        L \ll \frac{ \log\abs{\sigma_1}^{(d-1)c} }{ \abs{y}^d \abs{\sigma_1}^{ \frac{1}{d-1} } } \cdot \frac{ \abs{\sigma_1} }{ \abs{\sigma_1} \abs{y}^{-1} } = \frac{ \log\abs{\sigma_1}^{(d-1)c} }{ \abs{y}^{d-1} \abs{\sigma_1}^{ \frac{1}{d-1} } }.
                    \]
                \end{enumerate}

                \item Case $j \not \in \set{1,d}$:
                \begin{enumerate}
                    \item Case $\abs{\sigma_j} \gg \abs{\sigma_1}^{\frac{1}{2}}$: We choose $k = d$, which fulfils \eqref{eq:cond}, and any other $l$ that does too. Then, $\abs{\sigma_d-\sigma_l} \ll \abs{\sigma_1}$, and $\abs{\sigma_j-\sigma_d} \abs{\sigma_j-\sigma_l} \gg \abs{\sigma_j}^2 \gg \abs{\sigma_1}$, as a combination of \eqref{eq:lb_trick} and the (sub-) case condition. Plugging everything into \eqref{eq:L-bound1} gives
                    \[
                        L \ll \frac{ \log\abs{\sigma_1}^{(d-1)c} }{ \abs{y}^d \abs{\sigma_1}^{ \frac{1}{d-1} } } \cdot \frac{ \abs{\sigma_1} }{ \abs{\sigma_1 }} = \frac{ \log\abs{\sigma_1}^{(d-1)c} }{ \abs{y}^d \abs{\sigma_1}^{ \frac{1}{d-1} } }.
                    \]

                    \item Case $\abs{\sigma_j} \ll \abs{\sigma_1}^{\frac{1}{2}}$: We choose $k = 1$, which fulfils \eqref{eq:cond}, and any other $l$ that does too. Then, $\abs{\sigma_1-\sigma_l} \ll \abs{\sigma_1}$, and $\abs{\sigma_j-\sigma_1} \gg \abs{\sigma_1}$, while $\abs{\sigma_j-\sigma_l} \gg \abs{\sigma_j} \gg \abs{y}^{-1}$, as a combination of \eqref{eq:lb_trick} and \eqref{eq:lb_y_sigj}.

                    Thus, 
                    \[
                        L \ll \frac{ \log\abs{\sigma_1}^{(d-1)c} }{ \abs{y}^d \abs{\sigma_1}^{ \frac{1}{d-1} } } \cdot \frac{ \abs{\sigma_1} }{ \abs{\sigma_1} \abs{y}^{-1} } = \frac{ \log\abs{\sigma_1}^{(d-1)c} }{ \abs{y}^{d-1} \abs{\sigma_1}^{ \frac{1}{d-1} } }.
                    \]
                \end{enumerate}
            \end{enumerate}

            In all four cases, we have at least that
            \begin{equation}\label{eq:L-bound2}
                L \ll \frac{ \log\abs{\sigma_1}^{(d-1)c} }{ \abs{y}^{d-1} \abs{\sigma_1}^{ \frac{1}{d-1} } },
            \end{equation}
            and we now show that this bound is indeed exponentially small (in $t$). To that end, we write
            \[
               \log\abs{\sigma_i} = \log\abs{\Tilde{\sigma_i}\br{ \gamma_1^{t_1} \cdots \gamma_s^{t_s} }} = t_1 \log\abs{\gamma_1^{(i)}} + \cdots + t_s \log\abs{\gamma_s^{(i)}},
            \]
            which gives the system of linear equations
            \[
                \underbrace{
                \begin{pmatrix}
                    \log\abs{\gamma_1^{(1)}} & \cdots & \log\abs{\gamma_s^{(1)}} \\
                    \vdots & \ddots & \vdots \\
                    \log\abs{\gamma_1^{(d)}} & \cdots & \log\abs{\gamma_s^{(d)}} 
                \end{pmatrix}
                }_{=: \Gamma}
                \begin{pmatrix}
                    t_1 \\ \vdots \\ t_s
                \end{pmatrix}
                =
                \begin{pmatrix}
                    \log\abs{\sigma_1} \\ \vdots \\ \log\abs{\sigma_d}
                \end{pmatrix}
            \]
            if done so for all $i \in  \set{1,\dots, d}$. We take the maximum norm and use the consistency $\norm{A x} \leq \norm{A} \cdot \norm{x}$ to get 
            \begin{equation}\label{eq:lsig_1-ub}
                \log\abs{\sigma_1} \leq c_1 t,
            \end{equation}
            where $c_1 = \norm{\Gamma}_{\max}$ is effective and does not depend on $\br{x,y\br{t_1,\dots, t_s}}$. 
            
            Similarly, since $\gamma_1, \dots, \gamma_s$ are multiplicatively independent, the matrix $\Gamma$ has full column rank. Thus, $\Gamma^\top \Gamma$ is invertible, we multiply with $\Gamma^\top$ and $\br{\Gamma^\top \Gamma}^{-1}$ and take the maximum norm. This gives $t \leq c_2 \log\abs{\sigma_1}$ or $\abs{\sigma_1} \geq e^{ \frac{1}{c_2} t} $. 

            We plug in the upper and lower bounds for $\abs{\sigma_1}$ into our bound \eqref{eq:L-bound2} for $L$ and get, by basically ignoring the contribution of $\abs{y}$ with $\abs{y} \geq 1$,
            \begin{equation}\label{eq:L-bound3}
                L \ll \frac{ \log\abs{\sigma_1}^{(d-1)c} }{ \abs{y}^{d-1} \abs{\sigma_1}^{ \frac{1}{d-1} } } \ll \frac{ \br{c_1 t}^{(d-1)c} }{ \abs{y}^{d-1} e^{ \frac{1}{c_2}t \frac{1}{d-1} } } \ll e^{-c_3 t}
            \end{equation}
            for some effectively computable constant $c_3$.
            
            We now return to Siegel's Identity \eqref{eq:Siegel}, apply the bound from Equation \eqref{eq:L-bound3} and get $\log L' = \log \abs{1-L} \ll e^{-c_3 t}$. Also note that $L'\neq 1$, since $L = \frac{\beta_j}{\beta_k} \cdot \frac{\sigma_k-\sigma_l}{\sigma_j-\sigma_l} = 0$ would imply that $\sigma_k = \sigma_l$ but this is impossible per our requirement that $\Q\br{\gamma_1^{t_1}\cdots \gamma_s^{t_s}} = K$. Thus,
            \[
                0 \neq \abs{ \log L' } = \abs{ \log\abs{ \frac{\beta_l}{\beta_k} } + \log\abs{ \frac{\sigma_j-\sigma_k}{\sigma_j-\sigma_l} } } \ll e^{- c_3 t}.
            \]
            Let us now call $\sigma_A = \max\br{ \abs{\sigma_j}, \abs{\sigma_k} }$, $\sigma_a = \min\br{ \abs{\sigma_j}, \abs{\sigma_k} }$ and $\sigma_B = \max\br{\abs{\sigma_j}, \abs{\sigma_l}},$ $\sigma_b = \min\br{\abs{\sigma_j}, \abs{\sigma_l}} $. Then
            \begin{align*}
                \log\abs{ \frac{\sigma_j-\sigma_k}{\sigma_j-\sigma_l} } = \log\frac{\sigma_A}{\sigma_B} + \log\abs{ \frac{ 1 - \frac{\sigma_a}{\sigma_A} }{ 1 - \frac{\sigma_b}{\sigma_B} } },
            \end{align*}
            and since $k,l$ fulfil \eqref{eq:cond}, both $\frac{\sigma_a}{\sigma_A}, \frac{\sigma_b}{\sigma_B} \leq t^{-\kappa}$. Thus, $\log\abs{ \frac{ 1 - \frac{a}{A} }{ 1 - \frac{b}{B} } } = O\br{ t^{-\kappa} }$, which gives
            \[
                \Lambda = \abs{ \log\abs{ \frac{\beta_l}{\beta_k} } + \log\frac{\sigma_A}{\sigma_B} } \ll t^{-\kappa}.
            \]
            
            Assume for now that $\Lambda \neq 0$. Let $r$ be the unit rank of our number field then we have $\beta_k = \br{ \eta_1^{(k)} }^{b_1} \cdots \br{ \eta_r^{(k)} }^{b_r}$ in terms of the fundamental units $\eta_1, \dots, \eta_r$ of $\Z_K^\times$, same for $\beta_l$, while we can write $\sigma_A = \br{\gamma_1^{(A)}}^{t_1}\cdots \br{\gamma_s^{(A)}}^{t_s}$, same for $\sigma_B$. We can thus write
            \begin{equation}\label{eq:linearform}
                \Lambda = \abs{ \sum_{i=1}^r b_i \br{ \log\abs{ \eta_i^{(l)} } - \log\abs{ \eta_i^{(k)} } } + \sum_{i=1}^s t_i \br{ \log\abs{ \gamma_i^{(A)} }-\log\abs{ \gamma_i^{(B)} } } }  \ll t^{-\kappa}.
            \end{equation}
            
            We now argue that we can bound the $b_i$ and thus all coefficients of $\Lambda$ by $t$ to then apply Proposition \ref{prop:Baker}. 

            To that end, note that for any $i \neq j$, we have that
            \begin{align*}
                \log\abs{\beta_i} &= \log\abs{ x - \sigma_j y + y\br{ \sigma_j - \sigma_i } } \\
                &= \log\abs{y} + \log\abs{ \sigma_i - \sigma_j } + \log\abs{ 1 + \frac{\beta_j}{y \abs{\sigma_i-\sigma_j}} } \\
                &\ll \log\abs{y} + \log\abs{\sigma_i - \sigma_j} .
            \end{align*}

            On the one hand, we have $\abs{\sigma_i - \sigma_j} \leq 2 \max\br{ \abs{\sigma_i}, \abs{\sigma_j} } \ll \abs{\sigma_1}$. On the other hand, we have by Lemma \ref{lem:tijd} and Equation \eqref{eq:lb_y_sigj} that
            \[
                \abs{\sigma_i - \sigma_j} \geq \max\br{ \abs{\sigma_i}, \abs{\sigma_j} } - \min\br{ \abs{\sigma_i}, \abs{\sigma_j} }  > \frac{ \max\br{\abs{\sigma_i}, \abs{\sigma_j}} }{ h(\sigma_i)^{-c} } \gg \frac{ 1 }{ \abs{y} h(\sigma_i)^{c} }.
            \]
            We thus have
            \[
                \abs{ \log\abs{\sigma_i-\sigma_j} } \ll \max\br{ \log\abs{\sigma_1}, \log\abs{y} + \log h(\sigma_i) },
            \]
            where we have, as a reminder, $h(\sigma_i) \ll \log\abs{\sigma_1}$ by looking at the Mahler measure to compute the height, and $\log\abs{\sigma_1} \ll t$ by Equation \eqref{eq:lsig_1-ub}. This gives $\log\abs{\sigma_i - \sigma_j} \ll \log y + t$ and thus
            \begin{align*}
                \log\abs{\beta_i} \ll \log\abs{y} + t
            \end{align*}
            for all $i \neq j$. We also get the same bound for the positive quantity $-\log\abs{\beta_j}$, since by the above inequality,
            \[
                -\log\abs{\beta_j} = \sum_{i = 1, i \neq j}^d \log\abs{\beta_i} \ll \log\abs{y} + t.
            \]

            Next, we take a look at the coefficients of the polynomial $\br{X - \sigma_1}\cdots \br{X - \sigma_d}$. Their absolute values can obviously be bounded by $\abs{\sigma_1}^d$ and $\abs{\sigma_1}^d \leq e^{d\, c_1 t}$ by Equation \eqref{eq:lsig_1-ub}. If we then apply Proposition \ref{prop:BuGy} with $H = e^{d\, c_1 t}$, since $R \ll 1$ as it does not depend on the $t_i$, we get that
            \begin{equation}\label{eq:logxy-bound}
                \log\abs{x}, \log\abs{y} \ll t
            \end{equation}
            and thus
            \begin{equation}\label{eq:bugy-bound}
                \abs{ \log\abs{\beta_i} } \ll \log\abs{y} + t \ll t
            \end{equation}
            for all $i \in \set{1, \dots, d}$.

            We return to our decomposition into powers of fundamental units,
            \[
                \beta_i = \br{ \eta_1^{(i)} }^{b_1} \cdots \br{ \eta_r^{(i)} }^{b_r}.
            \]
            Doing this for all $i = 1, \dots, r$, and we only care that the unit rank $r < d$, reveals that $(b_1, \dots, b_r)$ is a solution to the system of linear equations
            \[
                \underbrace{
                \begin{pmatrix}
                    \log\abs{ \eta_1^{(1)} } & \cdots & \log\abs{ \eta_r^{(1)} } \\
                    \vdots & \ddots & \vdots \\
                    \log\abs{ \eta_1^{(r)} } & \cdots & \log\abs{ \eta_r^{(r)} }
                \end{pmatrix}
                }_{= H}
                \begin{pmatrix}
                    b_1 \\ \vdots \\ b_r
                \end{pmatrix}
                =
                \begin{pmatrix}
                    \log\abs{\beta_1} \\ \vdots \\ \log\abs{\beta_r}
                \end{pmatrix}.
            \]

            Since the $\eta_i$ are multiplicatively independent, the matrix $H$ is invertible. We multiply with the inverse $H^{-1}$ and apply the maximum norm, which gives, in combination with Equation \eqref{eq:bugy-bound},
            \begin{equation}\label{eq:bi-bound}
                \max_i\abs{b_i} \leq c_4 \max_i \log\abs{\beta_i} \ll t,
            \end{equation}
            where $c_4 = \norm{H^{-1}}_{\max}$. 

            If we return to Equation \eqref{eq:linearform}, we now have effectively bounded the absolute value of every coefficient by $t$. The heights of the $\eta_i^{(\cdot)}, \gamma_i^{(\cdot)}$ do not depend on $x,y$ or the $t_i$ and are thus effectively bounded by $1$. In the case that $\Lambda \neq 0$, we plug everything into Proposition \ref{prop:Baker}, with $h_i \ll 1, B \ll t$ and get $\log\abs{\Lambda} \geq c_5 \log t$ for some effectively computable constant $c_5$. If we compare this with the logarithm of the upper bound from Equation \eqref{eq:linearform}, we get that $\kappa \log t \leq c_5 \log t$. 
            
            Now, if $\kappa$ is sufficiently large, i.e. larger than the constant $c_5$ which itself is independent from $\kappa$, we get that $t \ll 1$.
            
            We now have to check what happens if the linear form vanishes instead, i.e. if $\Lambda = 0$. In this case, we have that $\frac{\beta_l}{\beta_k} = \frac{\sigma_B}{\sigma_A}$ and we have to differentiate between four cases:

            \begin{enumerate}
                \item Case $A=B=j$: This implies $\beta_l = \beta_k$ and thus $\sigma_l = \sigma_k$, which we again ruled out by requiring that $\Q\br{\gamma_1^{t_1}\cdots \gamma_s^{t_s}} = K$.

                \item Case $A=k, B = l$: This implies $\beta_l\sigma_k = \beta_k \sigma_l$, if we plug this into Siegel's Identity \eqref{eq:Siegel-alt}, it then gives $\beta_j\br{\sigma_k-\sigma_l} = \beta_l\sigma_j + \beta_k\sigma_j$ and thus
                \[
                    \frac{\beta_j \br{ \sigma_k-\sigma_l } }{\beta_l \sigma_j} - 1 = \frac{\beta_k}{\beta_l}.
                \]
                The fraction on the left-hand side is $\ll e^{- c_6 t}$, which follows completely analogously to how we showed Equation \eqref{eq:L-bound3}, that $L \ll e^{-c_3 t}$. Thus,
                \[
                    \log\abs{ \frac{\beta_k}{\beta_l} } = \log\br{ 1 - O\br{e^{-c_6 t}} } \ll e^{-c_6 t},
                \]
                and the first equality also implies that $\log\abs{\beta_k / \beta_l} \neq 0$. We now do the same thing, i.e. write $\beta_i = \br{ \eta_1^{(i)} }^{b_1} \cdots \br{ \eta_r^{(i)} }^{b_r}$ and apply Proposition \ref{prop:Baker} to the linear form $0 \neq \log\abs{\beta_k / \beta_l} \ll e^{-c_6 t}$, since $\abs{b_i} \ll t$ by Equation \eqref{eq:bi-bound}, and deduce $t \ll 1$.

                \item Case $A=k, B = j$: First, this means that $\abs{\sigma_k} > \abs{\sigma_j} > \abs{\sigma_l}$. Second, this implies $\beta_l\sigma_k = \beta_k\sigma_j$ and thus
                $\beta_j\br{ \sigma_k-\sigma_l } + \beta_l\sigma_j + \beta_k\sigma_l - 2 \beta_k\sigma_j = 0$, if we plug it into Siegel's identity \eqref{eq:Siegel-alt}. This gives
                \[
                    \frac{\beta_j \br{\sigma_k-\sigma_l}}{2\beta_k\sigma_j} + \frac{\sigma_l}{2 \sigma_j} - 1 = -\frac{\beta_l}{2\beta_k}.
                \]
                
                The first fraction on the left-hand side is $\ll e^{-c_7 t}$, analogously to how we showed $L \ll e^{-c_3 t}$. The second fraction $ \abs{\sigma_l / \sigma_j} \ll t^{-\kappa}$, since $l$ fulfils \eqref{eq:cond} and $\abs{\sigma_j} \geq \abs{\sigma_l}$. This gives
                \[
                    0 \neq \abs{ \log\abs{ \frac{\beta_l}{\beta_k} } - \log 2 } \ll t^{- \kappa},
                \]
                and thus $t \ll 1$ analogously to the case $\Lambda \neq 0$.

                \item Case $A=j, B = l$: This is analogous to the previous case, we now have $\abs{\sigma_l} \geq \abs{\sigma_j} \geq \abs{\sigma_k}$ and $ \beta_l\sigma_j = \beta_k\sigma_l $ and get, by plugging this into Siegel's identity \eqref{eq:Siegel},
                \[
                    -\frac{\beta_j\br{\sigma_k-\sigma_l}}{2\beta_l\sigma_j} + \frac{\sigma_k}{2\sigma_j} - 1 = -\frac{\beta_k}{2\beta_l}.
                \]
                The first fraction on the left-hand side is again $\ll e^{-c_8 t}$, analogously to $L \ll e^{-c_3 t}$, while the second fraction is $\ll t^{-\kappa}$, since $k$ fulfils \eqref{eq:cond} and $\abs{\sigma_j} \geq \abs{\sigma_k}$. This gives
                \[
                    0 \neq \abs{ \log\abs{ \frac{\beta_k}{\beta_l} } - \log 2 } \ll t^{-\kappa}
                \]
                and thus $t \ll 1$ analogously to the case $\Lambda \neq 0$.
                
            \end{enumerate}

        We have proven $t \ll 1$ in all four subcases and thus throughout Case 1. If we plug this into Equation \eqref{eq:logxy-bound}, we have $\abs{x}, \abs{y} \ll 1$ and thus an effective upper bound to the size of the solutions, which means that there are only finitely many.
            
        \subsection{Case 2}
            If \eqref{eq:cond} does not hold, then $\log\abs{\sigma_i / \sigma_j} \ll \log t$ holds instead for all but one index $i \in \set{1, \dots, d}\backslash\set{j}$. We rename the indices so it holds for $ i = 1, \dots, d-2$, while $j=d$. We lose the ordering of $\sigma_1, \dots, \sigma_d$ but do not need this in the following arguments. If we rewrite $ \log\abs{\sigma_i / \sigma_d} \ll \log t $ for $i \in \set{1, \dots, d-2}$, then this means that
            \begin{equation}\label{eq:Case2}
                \underbrace{
                \begin{pmatrix}
                    \log\abs{ \frac{ \gamma_1^{(1)} }{ \gamma_1^{(d)} } } & \cdots & \log\abs{ \frac{ \gamma_s^{(1)} }{ \gamma_s^{(d)} } } \\
                    \vdots & \ddots & \vdots \\
                    \log\abs{ \frac{ \gamma_1^{(d-2)} }{ \gamma_1^{(d)} } } & \cdots & \log\abs{ \frac{ \gamma_s^{(d-2)} }{ \gamma_s^{(d)} } }
                \end{pmatrix}
                }_{= \Gamma}
                \begin{pmatrix}
                    t_1 \\ \vdots \\ t_s
                \end{pmatrix}
                \ll
                \begin{pmatrix}
                    \log t \\ \vdots \\ \log t
                \end{pmatrix}
            \end{equation}
            holds. The matrix $\Gamma$ has full column rank per our requirement in our theorem, thus $\Gamma^\top \Gamma$ is invertible. We multiply both sides first with $\Gamma^\top$ and then $\br{\Gamma^\top\Gamma}^{-1}$ and apply the maximum norm. 
            
            Since the matrix does not depend on $\br{x,y,\br{t_1,\dots, t_s}}$, neither does their norm, and after using $\norm{Ax} \leq \norm{A}\cdot \norm{x}$, we get $t \ll \log t$ and thus $t \ll 1$.

            If we plug this into Equation \eqref{eq:logxy-bound}, we can use Proposition \ref{prop:BuGy} in Case 2 as well to derive said inequality, we get that $\abs{x}, \abs{y} \ll 1$, which proves the finiteness of the number of solutions to Thue Equation \eqref{eq:main}.

        \section{Connection to Schanuel's conjecture}
    
            If we assume Schanuel's conjecture, then the full rank of the matrix $\Gamma$ follows from the $\Q$-linear independency of $\gamma_1, \dots, \gamma_s$. For sake of completeness, we state this as a second Theorem:

            \begin{theorem}
                Let $K$ be a number field of degree $d \geq 3$ and $s \leq d-2$. Let $\gamma_1, \dots \gamma_s \in K^*$ be multiplicatively independent and $\Q$-linearly independent algebraic integers.

                If Schanuel's conjecture holds, then the Thue equation
                \begin{equation*}
                    \abs{N_{K/\Q}\br{X-\gamma_1^{t_1}\cdots\gamma_s^{t_s}Y}} = 1    
                \end{equation*}
                has only finitely many integer solutions $\br{x, y, \br{t_1, \dots, t_s}} \in \Z^2 \times \N^s$ with $xy \neq 0$ and $\Q\br{\gamma_1^{t_1}\cdots\gamma_s^{t_s}} = K$, all of which can be effectively computed.
            \end{theorem}

            \begin{proof}

                The proof for Case 1 of Theorem 1 can be carried over par for par. We adapt the proof for Case 2, i.e. that Equation \eqref{eq:Case2} holds.
                
                Assume that the matrix $\Gamma$ does not have full rank, that there exist $x_1, \dots, x_s$, not all zero such that
                \[
                    x_1
                    \begin{pmatrix}
                        \log\abs{ \frac{ \gamma_1^{(1)} }{ \gamma_1^{(d)} } } \\ \vdots \\ \log\abs{ \frac{ \gamma_1^{(d-2)} }{ \gamma_1^{(d)} } }    
                    \end{pmatrix}
                    + \cdots +
                    x_s
                    \begin{pmatrix}
                        \log\abs{ \frac{ \gamma_{s\phantom{1}}^{(1)} }{ \gamma_{s\phantom{1}}^{(d)} } } \\ \vdots \\ \log\abs{ \frac{ \gamma_{s\phantom{1}}^{(d-2)} }{ \gamma_{s\phantom{1}}^{(d)} } }    
                    \end{pmatrix}
                    =
                    \begin{pmatrix}
                        0 \\ \vdots \\ 0
                    \end{pmatrix}.
                \]
                If we rewrite this slightly, then this means that
                \[
                    \begin{pmatrix}
                        \log\abs{\gamma_1^{(1)}} & \cdots & \log\abs{\gamma_s^{(1)}} \\
                        \vdots & \ddots & \vdots \\
                        \log\abs{\gamma_1^{(d-2)}} & \cdots & \log\abs{\gamma_s^{(d-2)}}
                    \end{pmatrix}
                    \begin{pmatrix}
                        x_1 \\ \vdots \\ x_s
                    \end{pmatrix}
                    =
                    \begin{pmatrix}
                        \lambda \\ \vdots \\ \lambda 
                    \end{pmatrix},
                \]
                where $\lambda = x_1 \log\abs{\gamma_1^{(d)}} + \dots + x_s \log\abs{\gamma_s^{(d)}} \neq 0$, since the $\gamma_i$ are multiplicatively independent. Thus $\lambda$ is a non-zero eigenvalue of the matrix on the left-hand side, which in turn means that $\lambda$ and $\log\abs{\gamma_1^{(i)}}, \dots, \log\abs{\gamma_s^{(i)}}$ are algebraically dependent. But Schanuel's conjecture asserts that if $\gamma_1, \dots, \gamma_s$ are linearly independent over $\Q$ then $\log\abs{\gamma_1}, \dots, \log\abs{\gamma_s}$ are algebraically independent over $\overline{\Q}$. This gives the contradiction and thus the full rank for the original matrix $\Gamma$.

                We can then proceed analogously to the proof of Theorem 1 and derive $t \ll 1$ and in further consequence $\abs{x}, \abs{y} \ll 1$.

            \end{proof}
  
    \printbibliography

@Article{Bomb82,
	Author = {E. Bombieri},
	Title = {{On the Thue-Siegel-Dyson theorem}},
	FJournal = {{Journal f\"ur die Reine und Angewandte Mathematik}},
	Journal = {{Acta Math}},
	Volume = {148},
	Pages = {255--296},
	Year = {1982},
	DOI = {10.1007/BF02392731}
}

@inbook{Smyth08,
    series={London Mathematical Society Lecture Note Series},
    title={The Mahler measure of algebraic numbers: a survey},
    DOI={10.1017/CBO9780511721274.021},
    booktitle={Number Theory and Polynomials},
    publisher={Cambridge University Press},
    author={Smyth, Chris},
    editor={McKee, James and Smyth, ChrisEditors},
    year={2008},
    pages={322–349},
    collection={London Mathematical Society Lecture Note Series}
}

@Article{BaWh93,
	Author = {A. {Baker} and G. {W\"ustholz}},
	Title = {{Logarithmic forms and group varieties}},
	FJournal = {{Journal f\"ur die Reine und Angewandte Mathematik}},
	Journal = {{J. Reine Angew. Math.}},
	ISSN = {0075-4102},
	Volume = {442},
	Pages = {19--62},
	Year = {1993},
	Publisher = {De Gruyter, Berlin},
	DOI = {10.1515/crll.1993.442.19},
	MSC2010 = {11J86 14L10},
	Zbl = {0788.11026}
}

@Article{BuGy96,
	Author = {Yann {Bugeaud} and K\'alm\'an {Gy\H{o}ry}},
	Title = {{Bounds for the solutions of Thue-Mahler equations and norm form equations}},
	FJournal = {{Acta Arithmetica}},
	Journal = {{Acta Arith.}},
	ISSN = {0065-1036},
	Volume = {74},
	Number = {3},
	Pages = {273--292},
	Year = {1996},
	Publisher = {Polish Academy of Sciences (Polska Akademia Nauk - PAN), Institute of Mathematics (Instytut Matematyczny), Warsaw},
	DOI = {10.4064/aa-74-3-273-292},
	MSC2010 = {11D57 11J13},
	Zbl = {0861.11024}
}

@Article{Tijd73,
    author = {Tijdeman, R.},
    journal = {Compositio Mathematica},
    number = {3},
    pages = {319-330},
    publisher = {Noordhoff International Publishing},
    title = {On integers with many small prime factors},
    url = {http://eudml.org/doc/89173},
    volume = {26},
    year = {1973},
}

@Article{Tho90,
 Author = {Thomas, Emery},
 Title = {Complete solutions to a family of cubic {Diophantine} equations},
 FJournal = {Journal of Number Theory},
 Journal = {J. Number Theory},
 ISSN = {0022-314X},
 Volume = {34},
 Number = {2},
 Pages = {235--250},
 Year = {1990},
 DOI = {10.1016/0022-314X(90)90154-J},
 zbMATH = {4142110},
 Zbl = {0697.10011}
}

@Article{LeWa15,
 Author = {Levesque, Claude and Waldschmidt, Michel},
 Title = {A family of {Thue} equations involving powers of units of the simplest cubic fields},
 FJournal = {Journal de Th{\'e}orie des Nombres de Bordeaux},
 Journal = {J. Th{\'e}or. Nombres Bordx.},
 ISSN = {1246-7405},
 Volume = {27},
 Number = {2},
 Pages = {537--563},
 Year = {2015},
 DOI = {10.5802/jtnb.913},
 zbMATH = {6504492},
 Zbl = {1395.11059}
}

@Article{LeWa13,
 Author = {Levesque, Claude and Waldschmidt, Michel},
 Title = {Solving effectively some families of {Thue} {Diophantine} equations},
 FJournal = {Moscow Journal of Combinatorics and Number Theory},
 Journal = {Mosc. J. Comb. Number Theory},
 ISSN = {2220-5438},
 Volume = {3},
 Number = {3-4},
 Pages = {118--144},
 Year = {2013},
 URL = {mjcnt.phystech.edu/en/download.php?id=76},
 zbMATH = {6367617},
 Zbl = {1352.11035}
}

@article{LeWa12,
    author = {Levesque, Claude and Waldschmidt, Michel},
    journal = {Acta Arithmetica},
    language = {fre},
    number = {2},
    pages = {117-138},
    title = {Familles d'équations de Thue-Mahler n'ayant que des solutions triviales},
    url = {http://eudml.org/doc/279336},
    volume = {155},
    year = {2012}
}

@Book{Lan66,
 Author = {Lang, Serge},
 Title = {Introduction to Transcendental Numbers},
 Publisher = {Addison-Wesley},
 Pages = {30--31},
 Year = {1966},
 Language = {English}
}
\end{document}